\documentclass[12pt,reqno]{article}

\usepackage[usenames]{color}
\usepackage[colorlinks=true,
linkcolor=webgreen, filecolor=webbrown,
citecolor=webgreen]{hyperref}

\definecolor{webgreen}{rgb}{0,.5,0}
\definecolor{webbrown}{rgb}{.6,0,0}

\usepackage{amssymb}
\usepackage{graphicx}
\usepackage{amscd}
\usepackage{lscape}
\usepackage{tikz}
\usepackage{tikz-cd}
\usepackage{pgfplots}

\usetikzlibrary{matrix}
\usetikzlibrary{fit,shapes}
\usetikzlibrary{positioning, calc}
\tikzset{circle node/.style = {circle,inner sep=1pt,draw, fill=white},
        X node/.style = {fill=white, inner sep=1pt},
        dot node/.style = {circle, draw, inner sep=5pt}
        }
\usepackage{tkz-fct}

\usepackage{amsthm}
\newtheorem{theorem}{Theorem}

\newtheorem{proposition}[theorem]{Proposition}
\newtheorem{corollary}[theorem]{Corollary}
\newtheorem{conjecture}[theorem]{Conjecture}

\theoremstyle{definition}

\newtheorem{example}[theorem]{Example}

\usepackage{float}

\usepackage{graphics,amsmath}
\usepackage{amsfonts}
\usepackage{latexsym}
\usepackage{epsf}

\newcommand{\seqnum}[1]{\href{http://oeis.org/#1}{\underline{#1}}}

\begin{document}

\begin{center}
\vskip 1cm{\LARGE\bf The second production matrix of a Riordan array} \vskip 1cm \large
Paul Barry\\
School of Science\\
Waterford Institute of Technology\\
Ireland\\
\href{mailto:pbarry@wit.ie}{\tt pbarry@wit.ie}
\end{center}
\vskip .2 in

\begin{abstract} A Riordan array is defined by its production matrix. In this paper, we explore the notion of the second production matrix of a Riordan array, characterizing the matrix that is generated by it. We indicate how this procedure can be generalized to the so-called $n$-th production matrix of a Riordan array.
\end{abstract}

\section{Introduction}
A Riordan array $M$ \cite{book, SGWW} is a matrix representation of an element $(g(x), f(x)) \in \mathbb{C}[[x]] \times \mathbb{C}[[x]]$ where $$g(x)=g_0+g_1x+ g_2 x^2+ \cdots, \quad g_0 \ne 0,$$ and
$$f(x)=f_1 x+ f_2x^2+ f_3 x^3+\cdots \quad f_0=0, f_1 \ne 0.$$

Then we have $$M_{n,k}=[x^n] g(x) f(x)^k.$$ The stipulation on $f(x)$ ensures that this matrix is lower-triangular, and since $g_0 \ne 0$ as well, it is invertible. We note that $M=(M_{n,k})_{0\le n,k \le \infty}$ is infinite in extent. 

The set of pairs of power series as described above form a group (the Riordan group), with the operations
$$(g(x), f(x))\cdot (u(x), v(x))=(g(x)u(f(x)), v(f(x))$$
$$(g(x), f(x))^{-1}=\left(\frac{1}{g(\bar{f}(x))}, \bar{f}(x)\right).$$
Here, $\bar{f}(x)$ is the compositional inverse of $f(x)$. Thus we have $f(\bar{f}(x))=x$ and $\bar{f}(f(x))=x$. The power series $\bar{f}(x)$ is the solution $u$ to $f(u)=x$ that satisfies $u(0)=0$.

These group operations are translated to ordinary matrix multiplication and the matrix inverse in the Riordan array matrix representation.

The fundamental theorem of Riordan arrays (FTRA) says that the group action of the Riordan group on $\mathbb{C}[[x]]$ is given by $$(g(x), f(x))\cdot h(x)=g(x)h(f(x)).$$

To the Riordan array $A$ we associate the matrix
$$P_M= M^{-1} \overline{M},$$
where $\overline{M}$ is the matrix $M$ with its top row removed. This matrix is a lower Hessenberg matrix. $P_M$ is called the \emph{production matrix} of $M$ \cite{Deutsch, He}, and with a knowledge of $P_M$ we can recover $M$. Riordan arrays are characterized by the form of their production array. Thus a matrix $M$ is a Riordan array if and only if $P_M$ is of the form
$$\left(
\begin{array}{ccccccc}
 z_0 & a_0 & 0 & 0 & 0 & 0 & 0 \\
 z_1 & a_1 & a_0 & 0 & 0 & 0 & 0 \\
 z_2 & a_2 & a_1 & a_0 & 0 & 0 & 0 \\
 z_3 & a_3 & a_2 & a_1 & a_0 & 0 & 0 \\
 z_4 & a_4 & a_3 & a_2 & a_1 & a_0 & 0 \\
 z_5 & a_5 & a_4 & a_3 & a_2 & a_1 & a_0 \\
 z_6 & a_6 & a_5 & a_4 & a_3 & a_2 & a_1 \\
\end{array}
\right).$$
We define the power series $Z(x)$ by  $Z(x)=\sum_{n=0}^{\infty} z_n x^n$; the sequence $z_0, z_1, \ldots$ is called the $Z$-sequence of the Riordan array $M$. We similarly define the power series $A(x)$ by $A(x)=\sum_{n=0}^{\infty} a_n x^n$. The sequence $a_0, a_1,\ldots$ is called the $A$-sequence of the Riordan array $M$.

For simplicity, we shall assume in the sequel that $g_0=f_1=1$.

Given an $A$-sequence and a $Z$-sequence, then the corresponding Riordan array is given by
$$ M = \left(1-\frac{xZ(x)}{A(x)}, \frac{x}{A(x)}\right)^{-1}.$$
When restricted to integer entries (which will be the case of all examples in this note), many important Riordan arrays are the be found in the On-Line Encyclopedia of Integer Sequences (OEIS) \cite{SL1, SL2}. For instance, Pascal's triangle \seqnum{A007318} $(\binom{n}{k})_{0 \le n,k\le \infty}$ represents the Riordan group element $\left(\frac{1}{1-x},\frac{x}{1-x}\right)$.

\section{The second production matrix of a Riordan array}
We define the second production matrix of a Riordan array $M$ to be the matrix
$$ M^{-1} \overline{\overline{M}},$$ with its first column removed. We write this as
$$\tilde{P}_M=|M^{-1} \overline{\overline{M}},$$ with an obvious notation.
\begin{example}
We take the example of the Riordan array $M=\left(\binom{n+k}{2k}\right)$ \seqnum{A085478}, which is the matrix representation of the Riordan group element $\left(\frac{1}{1-x}, \frac{x}{(1-x)^2}\right)$. This Riordan array begins
$$\left(
\begin{array}{ccccccc}
 1 & 0 & 0 & 0 & 0 & 0 & 0 \\
 1 & 1 & 0 & 0 & 0 & 0 & 0 \\
 1 & 3 & 1 & 0 & 0 & 0 & 0 \\
 1 & 6 & 5 & 1 & 0 & 0 & 0 \\
 1 & 10 & 15 & 7 & 1 & 0 & 0 \\
 1 & 15 & 35 & 28 & 9 & 1 & 0 \\
 1 & 21 & 70 & 84 & 45 & 11 & 1 \\
\end{array}
\right),$$ and its (first) production matrix begins
$$\left(
\begin{array}{ccccccc}
 1 & 1 & 0 & 0 & 0 & 0 & 0 \\
 0 & 2 & 1 & 0 & 0 & 0 & 0 \\
 0 & -1 & 2 & 1 & 0 & 0 & 0 \\
 0 & 2 & -1 & 2 & 1 & 0 & 0 \\
 0 & -5 & 2 & -1 & 2 & 1 & 0 \\
 0 & 14 & -5 & 2 & -1 & 2 & 1 \\
 0 & -42 & 14 & -5 & 2 & -1 & 2 \\
\end{array}
\right).$$
The matrix $M^{-1} \overline{\overline{M}}$ begins
$$\left(
\begin{array}{ccccccc}
 1 & 3 & 1 & 0 & 0 & 0 & 0 \\
 0 & 3 & 4 & 1 & 0 & 0 & 0 \\
 0 & -2 & 2 & 4 & 1 & 0 & 0 \\
 0 & 4 & 0 & 2 & 4 & 1 & 0 \\
 0 & -10 & -1 & 0 & 2 & 4 & 1 \\
 0 & 28 & 4 & -1 & 0 & 2 & 4 \\
 0 & -84 & -14 & 4 & -1 & 0 & 2 \\
\end{array}
\right).$$
Thus the matrix $|M^{-1} \overline{\overline{M}}$  begins
$$\left(
\begin{array}{ccccccc}
 3 & 1 & 0 & 0 & 0 & 0 & 0 \\
 3 & 4 & 1 & 0 & 0 & 0 & 0 \\
 -2 & 2 & 4 & 1 & 0 & 0 & 0 \\
 4 & 0 & 2 & 4 & 1 & 0 & 0 \\
 -10 & -1 & 0 & 2 & 4 & 1 & 0 \\
 28 & 4 & -1 & 0 & 2 & 4 & 1 \\
 -84 & -14 & 4 & -1 & 0 & 2 & 4 \\
\end{array}
\right).$$
We see that this is of the form of an ordinary production matrix, and hence it will generate a Riordan array.
\end{example}
This motivates the question: Can we characterize the matrix produced by the second production matrix $\tilde{P}_M$, in terms of the original matrix $M$?

\section{The principal result}
We answer the above question as follows.
\begin{proposition} The matrix $\tilde{M}$ produced by the second production matrix $\tilde{P}_M$ of a Riordan array $M$ is given by
$$\left(\frac{\bar{f}(x)}{x g(\bar{f}(x))}, \frac{\bar{f}(x)^2}{x}\right)^{-1}.$$
\end{proposition}
\begin{corollary}  The matrix $\tilde{M}$ produced by the second production matrix $\tilde{P}_M$ of a Riordan array $M$ is given by
$$\left(\frac{x}{f(x)}, \frac{x^2}{f(x)}\right)^{-1} \cdot (g(x), f(x)).$$
\end{corollary}
Thus $$\tilde{M}=N \cdot M,$$ where $N$ is the Riordan array defined by $\left(\frac{x}{f(x)}, \frac{x^2}{f(x)}\right)^{-1}$.

\begin{proof} We prove the corollary first.
We let
$$A = \left(\frac{x}{f(x)}, \frac{x^2}{f(x)}\right)^{-1} \cdot (g(x), f(x)).$$
The
\begin{align*}A^{-1}&= (g(x), f(x))^{-1}\cdot \left(\frac{x}{f(x)}, \frac{x^2}{f(x)}\right)\\
&= \left(\frac{1}{g(\bar{f})}, \bar{f}\right)\cdot \left(\frac{x}{f(x)}, \frac{x^2}{f(x)}\right)\\
&= \left(\frac{1}{g(\bar{f})} \frac{\bar{f}}{f(\bar{f}(x))}, \frac{\bar{f}^2}{f(\bar{f}(x))}\right)\\
&= \left(\frac{\bar{f}}{x g(\bar{f})}, \frac{\bar{f}^2}{x}\right).\end{align*} Thus we get that
$$\left(\frac{x}{f(x)}, \frac{x^2}{f(x)}\right)^{-1} \cdot (g(x), f(x))=\left(\frac{\bar{f}}{x g(\bar{f})}, \frac{\bar{f}^2}{x}\right)^{-1}.$$

We now turn to proving the proposition. For this, we look at the generating functions of the columns of $\overline{\overline{M}}$, and operate on them by $M^{-1}$.
The first column of $\overline{\overline{M}}$ has generating function
$$\frac{g(x)-1-g_1 x}{x^2}.$$ We use the FTRA to see what the effect of $(g, f)^{-1}=\left(\frac{1}{g(\bar{f})}, \bar{f}\right)$ is on it. We find that the resulting first column has generating function
$$\frac{1}{g(\bar{f})}\left(\frac{g(\bar{f})-1-g_1 \bar{f}}{\bar{f}^2}\right).$$
We actually discard this to produce $\tilde{P}_M$.

Subsequent columns of $\overline{\overline{M}}$ have generating functions
$$\frac{g(x)f(x)^k - x.0^{k-1}}{x^2},$$ and hence this leads to the column generating functions of $\tilde{P}_M$ being given by
$$\frac{1}{g(\bar{f})}\left(\frac{g(\bar{f})f(\bar{f})^k-\bar{f}0^{k-1}}{\bar{f}^2}\right)=\frac{x^k}{\bar{f}^2}-\frac{0^{k-1}}{\bar{f}g(\bar{f})}.$$ We then obtain 
$$\tilde{Z}(x)=\frac{x}{\bar{f}^2}-\frac{1}{\bar{f}g(\bar{f})}=\frac{1}{\bar{f}}\left(\frac{x}{\bar{f}}-\frac{1}{g(\bar{f})}\right),$$ and 
$$\tilde{A}(x)=\frac{x^2}{\bar{f}^2}.$$ 
Here, $\tilde{Z}$ and $\tilde{A}$ are the $Z$- and $A$-series of $\tilde{P}_M$. 
We then know that the matrix produced by $\tilde{P}_M$ is given by 
$$\left(1-\frac{x\tilde{Z}(x)}{\tilde{A}(x)}, \frac{x}{\tilde{A}(x)}\right)^{-1}.$$ 
Now we have 
\begin{align*}
\left(1-\frac{x\tilde{Z}(x)}{\tilde{A}(x)}, \frac{x}{\tilde{A}(x)}\right)&=\left(
1-\frac{x}{\left(\frac{x^2}{\bar{f}^2}\right)}.\frac{1}{\bar{f}}\left(\frac{x}{\bar{f}}-\frac{1}{g(\bar{f})}\right), \frac{x}{\left(\frac{x^2}{\bar{f}^2}\right)}\right)\\
&=\left(1-\frac{x\bar{f}^2}{x^2}\frac{1}{\bar{f}}\left(\frac{x}{\bar{f}}-\frac{1}{g(\bar{f})}\right), \frac{x \bar{f}^2}{x^2}\right)\\
&=\left(1-\frac{\bar{f}}{x}\left(\frac{x}{\bar{f}}-\frac{1}{g(\bar{f})}\right), \frac{\bar{f}^2}{x}\right)\\
&=\left(1-\left(1-\frac{\bar{f}}{x g(\bar{f})}\right), \frac{\bar{f}^2}{x}\right)\\
&=\left(\frac{\bar{f}}{xg(\bar{f})}, \frac{\bar{f}^2}{x}\right).\end{align*}

\end{proof}

\section{Examples}
\begin{example} We return to the initial example where we had $M=\left(\binom{n+k}{2k}\right)$, which represents the Riordan group element $\left(\frac{1}{1-x}, \frac{x}{(1-x)^2}\right)$.  The second production matrix $\tilde{P}_M$ in that case began
$$\left(
\begin{array}{ccccccc}
 3 & 1 & 0 & 0 & 0 & 0 & 0 \\
 3 & 4 & 1 & 0 & 0 & 0 & 0 \\
 -2 & 2 & 4 & 1 & 0 & 0 & 0 \\
 4 & 0 & 2 & 4 & 1 & 0 & 0 \\
 -10 & -1 & 0 & 2 & 4 & 1 & 0 \\
 28 & 4 & -1 & 0 & 2 & 4 & 1 \\
 -84 & -14 & 4 & -1 & 0 & 2 & 4 \\
\end{array}
\right).$$
Here, we have $f(x)=\frac{x}{(1-x)^2}$, which gives us $\bar{f}(x)=1-c(-x)=xc(-x)^2$, where 
$$c(x)=\frac{1-\sqrt{1-4x}}{2x},$$ 
the generating function of the Catalan numbers $C_n=\frac{1}{n+1}\binom{2n}{n}$ \seqnum{A000108}. We have $g(x)=\frac{1}{1-x}$.
We find that 
$$\frac{\bar{f}}{x g(\bar{f})}=\frac{(1+x)\sqrt{1+4x}-1-3x}{2x^3},$$ and 
$$\frac{\bar{f}(x)^2}{x}=\frac{1+4x+2x^2-(1+2x)\sqrt{1+4x}}{2x^3}=xc(-x)^4.$$ 
By the proposition, the produced matrix is now equal to the matrix representation of the Riordan group element
$$\left(\frac{(1+x)\sqrt{1+4x}-1-3x}{2x^3}, xc(-x)^4\right)^{-1}.$$ 
More simply, this is equal to the matrix given by 
$$\left((1-x)^2, x(1-x)^2\right)^{-1} \cdot \left(\frac{1}{1-x}, \frac{x}{(1-x)^2}\right).$$ 
We obtain
\begin{scriptsize}
$$\left(
\begin{array}{cccccc}
 1 & 0 & 0 & 0 & 0 & 0 \\
 3 & 1 & 0 & 0 & 0 & 0 \\
 12 & 7 & 1 & 0 & 0 & 0 \\
 55 & 42 & 11 & 1 & 0 & 0 \\
 273 & 245 & 88 & 15 & 1 & 0 \\
 1428 & 1428 & 627 & 150 & 19 & 1 \\
\end{array}
\right)=\left(
\begin{array}{cccccc}
 1 & 0 & 0 & 0 & 0 & 0 \\
 2 & 1 & 0 & 0 & 0 & 0 \\
 7 & 4 & 1 & 0 & 0 & 0 \\
 30 & 18 & 6 & 1 & 0 & 0 \\
 143 & 88 & 33 & 8 & 1 & 0 \\
 728 & 455 & 182 & 52 & 10 & 1 \\
\end{array}
\right)\cdot \left(
\begin{array}{cccccc}
 1 & 0 & 0 & 0 & 0 & 0 \\
 1 & 1 & 0 & 0 & 0 & 0 \\
 1 & 3 & 1 & 0 & 0 & 0 \\
 1 & 6 & 5 & 1 & 0 & 0 \\
 1 & 10 & 15 & 7 & 1 & 0 \\
 1 & 15 & 35 & 28 & 9 & 1 \\
\end{array}
\right).$$
\end{scriptsize}
We find that in this case that we have
$$\tilde{M}_{n,k}=\sum_{i=0}^n \frac{2i+2}{3n+2-i} \binom{3n+2-i}{n-i}\binom{i+k}{2k}.$$ 
\end{example}

\section{The third production matrix}
In this section, we continue our previous example. 
\begin{example} In this example, we start again with 
$M=\left(\binom{n+k}{2k}\right)$. This time, using an obvious notation, we look at 
$$\tilde{\tilde{P}}_M=||M^{-1} \overline{\overline{\overline{M}}},$$ and ask how the array generated by $\tilde{\tilde{P}}_M$ relates to $M$ (and to $\tilde{M}$). 
Calculating, we find that $M^{-1} \overline{\overline{\overline{M}}}$ begins
$$\left(
\begin{array}{ccccccc}
 1 & 6 & 5 & 1 & 0 & 0 & 0 \\
 0 & 4 & 10 & 6 & 1 & 0 & 0 \\
 0 & -3 & 0 & 9 & 6 & 1 & 0 \\
 0 & 6 & 5 & 2 & 9 & 6 & 1 \\
 0 & -15 & -14 & 0 & 2 & 9 & 6 \\
 0 & 42 & 41 & 0 & 0 & 2 & 9 \\
 0 & -126 & -126 & -1 & 0 & 0 & 2 \\
\end{array}
\right).$$ 
Then $\tilde{\tilde{P}}_M=||M^{-1} \overline{\overline{\overline{M}}}$ begins
$$\left(
\begin{array}{cccccccc}
 5 & 1 & 0 & 0 & 0 & 0 & 0 & 0 \\
 10 & 6 & 1 & 0 & 0 & 0 & 0 & 0 \\
 0 & 9 & 6 & 1 & 0 & 0 & 0 & 0 \\
 5 & 2 & 9 & 6 & 1 & 0 & 0 & 0 \\
 -14 & 0 & 2 & 9 & 6 & 1 & 0 & 0 \\
 41 & 0 & 0 & 2 & 9 & 6 & 1 & 0 \\
 -126 & -1 & 0 & 0 & 2 & 9 & 6 & 1 \\
 402 & 6 & -1 & 0 & 0 & 2 & 9 & 6 \\
\end{array}
\right).$$ 
This then generates the array $\tilde{\tilde{M}}$ that begins 
$$\left(
\begin{array}{cccccccc}
 1 & 0 & 0 & 0 & 0 & 0 & 0 & 0 \\
 5 & 1 & 0 & 0 & 0 & 0 & 0 & 0 \\
 35 & 11 & 1 & 0 & 0 & 0 & 0 & 0 \\
 285 & 110 & 17 & 1 & 0 & 0 & 0 & 0 \\
 2530 & 1100 & 221 & 23 & 1 & 0 & 0 & 0 \\
 23751 & 11165 & 2635 & 368 & 29 & 1 & 0 & 0 \\
 231880 & 115192 & 30345 & 5106 & 551 & 35 & 1 & 0 \\
 2330445 & 1206348 & 344318 & 66010 & 8729 & 770 & 41 & 1 \\
\end{array}
\right).$$ 
We find the following.
$$\tilde{\tilde{M}} = ((1-x)^4, x(1-x)^4)^{-1} \cdot \left(\frac{1}{1-x}, \frac{x}{(1-x)^2}\right).$$ 
\end{example}
In the general case, we have the following result.
\begin{proposition} We let $\tilde{\tilde{P}}_M=||M^{-1} \overline{\overline{\overline{M}}}$ be the third production matrix of the Riordan array $M$ given by $(g(x), f(x))$. Let $\tilde{\tilde{M}}$ denote the matrix generated by $\tilde{\tilde{P}}_M$. Then we have 
$$ \tilde{\tilde{M}} = \left(\left(\frac{x}{f(x)}\right)^2, x\left(\frac{x}{f(x)}\right)^2\right)^{-1}\cdot (g(x), f(x)).$$
\end{proposition}
\begin{proof}
The proof is similar to the case for $\tilde{M}$. The $A$ sequence for $\tilde{\tilde{M}}$ is derived by operating on $\frac{gf^3}{x^3}$ by the $\left(\frac{1}{g(\bar{f})}, \bar{f}\right)$ to give $\tilde{\tilde{A}}=\frac{x^3}{\bar{f}^3}$. We find that 
$$\tilde{\tilde{Z}}=\frac{1}{\bar{f}^2}\left(\frac{x^2}{\bar{f}}-\frac{\bar{f}}{g(\bar{f})}\right).$$ 
Then the inverse of $\tilde{\tilde{M}}$ is given by 
$$\left(1-\frac{x \tilde{\tilde{Z}}}{\tilde{\tilde{A}}}, \frac{x}{\tilde{\tilde{A}}}\right)=\left(\frac{\bar{f}^2}{x^2 g(\bar{f})}, \frac{\bar{f}^3}{x^2}\right).$$ 
Thus 
$$\tilde{\tilde{M}}=\left(\frac{\bar{f}^2}{x^2 g(\bar{f})}, \frac{\bar{f}^3}{x^2}\right)^{-1}=\left(\left(\frac{x}{f(x)}\right)^2, x\left(\frac{x}{f(x)}\right)^2\right)^{-1}\cdot (g(x), f(x)).$$
\end{proof}  
The following conjecture follows, where again we use an obvious notation.
\begin{conjecture} Let $P^{(n)}_M=|^{(n-1)}M^{-1}\overline{M}^{(n)}$ be the $n$-th production matrix of $M$.  Then the matrix $M^{(n)}$ produced by $P^{(n)}_M$ is given by
$$M^{(n)}=\left(\left(\frac{x}{f(x)}\right)^{n-1}, x\left(\frac{x}{f(x)}\right)^{n-1}\right)^{-1}\cdot (g(x), f(x)).$$
\end{conjecture}
\begin{example}
We take the case of 
$$M=(g(x), f(x))=(c(x), xc(x)),$$ which is the Catalan array $\left(\frac{k+1}{2n-k+1}\binom{2n-k+1}{n-k}\right)$ \seqnum{A033184} that begins
$$\left(
\begin{array}{ccccccc}
 1 & 0 & 0 & 0 & 0 & 0 & 0 \\
 1 & 1 & 0 & 0 & 0 & 0 & 0 \\
 2 & 2 & 1 & 0 & 0 & 0 & 0 \\
 5 & 5 & 3 & 1 & 0 & 0 & 0 \\
 14 & 14 & 9 & 4 & 1 & 0 & 0 \\
 42 & 42 & 28 & 14 & 5 & 1 & 0 \\
 132 & 132 & 90 & 48 & 20 & 6 & 1 \\
\end{array}
\right).$$ The second production matrix of this array begins 
$$\left(
\begin{array}{cccccc}
 2 & 1 & 0 & 0 & 0 & 0 \\
 3 & 2 & 1 & 0 & 0 & 0 \\
 4 & 3 & 2 & 1 & 0 & 0 \\
 5 & 4 & 3 & 2 & 1 & 0 \\
 6 & 5 & 4 & 3 & 2 & 1 \\
 7 & 6 & 5 & 4 & 3 & 2 \\
\end{array}
\right),$$ which produces the matrix $\tilde{M}=\frac{2k+1}{3n-k+2}\binom{3n-k+2}{n-k}$ \seqnum{A092276} that begins
$$\left(
\begin{array}{cccccc}
 1 & 0 & 0 & 0 & 0 & 0 \\
 2 & 1 & 0 & 0 & 0 & 0 \\
 7 & 4 & 1 & 0 & 0 & 0 \\
 30 & 18 & 6 & 1 & 0 & 0 \\
 143 & 88 & 33 & 8 & 1 & 0 \\
 728 & 455 & 182 & 52 & 10 & 1 \\
\end{array}
\right).$$
We have $\bar{f}=x(1-x)$ and $\frac{1}{g(\bar{f})}=1-x$. 
We find that 
$$\tilde{M}=\left((1-x)^2, x(1-x)^2\right)^{-1}=\left(\frac{1}{c(x)}, \frac{x}{c(x)}\right)\cdot (c(x), xc(x)).$$ We have that the third production matrix
$\tilde{\tilde{P}}_M$ begins 
$$\left(
\begin{array}{ccccccc}
 3 & 1 & 0 & 0 & 0 & 0 & 0 \\
 6 & 3 & 1 & 0 & 0 & 0 & 0 \\
 10 & 6 & 3 & 1 & 0 & 0 & 0 \\
 15 & 10 & 6 & 3 & 1 & 0 & 0 \\
 21 & 15 & 10 & 6 & 3 & 1 & 0 \\
 28 & 21 & 15 & 10 & 6 & 3 & 1 \\
 36 & 28 & 21 & 15 & 10 & 6 & 3 \\
\end{array}
\right).$$ The matrix $\tilde{\tilde{M}}$ that it produces then begins 
$$\left(
\begin{array}{ccccccc}
 1 & 0 & 0 & 0 & 0 & 0 & 0 \\
 3 & 1 & 0 & 0 & 0 & 0 & 0 \\
 15 & 6 & 1 & 0 & 0 & 0 & 0 \\
 91 & 39 & 9 & 1 & 0 & 0 & 0 \\
 612 & 272 & 72 & 12 & 1 & 0 & 0 \\
 4389 & 1995 & 570 & 114 & 15 & 1 & 0 \\
 32890 & 15180 & 4554 & 1012 & 165 & 18 & 1 \\
\end{array}
\right).$$ 
We find that 
$$\tilde{\tilde{M}}=\left((1-x)^3, x(1-x)^3\right)^{-1}=\left(\frac{1}{c(x)^2}, \frac{x}{c(x)^2}\right)^{-1}\cdot (c(x), x c(x)).$$
Proceeding, we find that $\tilde{\tilde{\tilde{P}}}_M=P^{(4)}_M$ begins 
$$\left(
\begin{array}{ccccccc}
 4 & 1 & 0 & 0 & 0 & 0 & 0 \\
 10 & 4 & 1 & 0 & 0 & 0 & 0 \\
 20 & 10 & 4 & 1 & 0 & 0 & 0 \\
 35 & 20 & 10 & 4 & 1 & 0 & 0 \\
 56 & 35 & 20 & 10 & 4 & 1 & 0 \\
 84 & 56 & 35 & 20 & 10 & 4 & 1 \\
 120 & 84 & 56 & 35 & 20 & 10 & 4 \\
\end{array}
\right),$$ and produces the array $M^{(4)}$ that begins 
$$\left(
\begin{array}{ccccccc}
 1 & 0 & 0 & 0 & 0 & 0 & 0 \\
 4 & 1 & 0 & 0 & 0 & 0 & 0 \\
 26 & 8 & 1 & 0 & 0 & 0 & 0 \\
 204 & 68 & 12 & 1 & 0 & 0 & 0 \\
 1771 & 616 & 126 & 16 & 1 & 0 & 0 \\
 16380 & 5850 & 1300 & 200 & 20 & 1 & 0 \\
 158224 & 57536 & 13485 & 2320 & 290 & 24 & 1 \\
\end{array}
\right).$$
We have 
$$M^{(4)}=\left((1-x)^4, x(1-x)^4\right)^{-1}=\left(\frac{1}{c(x)^3}, \frac{x}{c(x)^3}\right)\cdot (c(x), xc(x)).$$ 
\end{example}

\section{A family of orthogonal polynomials}
We finish this note by giving an example of the foregoing in the area of constant coefficient orthogonal polynomials defined by Riordan arrays \cite{classical}. For this, we recall that the binomial matrix $B=\left(\binom{n}{k}\right)$, otherwise known as Pascal's triangle, is given by
$$\left(\frac{1}{1-x}, \frac{x}{1-x}\right).$$ 
Then we have that $B^r$ is given by 
$$\left(\frac{1}{1-rx}, \frac{x}{1-rx}\right).$$ In fact, the collection of these matrices $B^r, r \in \mathbb{Z}$ forms a one-parameter subgroup of the Riordan group. 

\begin{proposition} 
The family $M(r)$ of matrices produced by the second production matrix of $B^r$ gives a one-parameter collection of moment arrays of a family of orthogonal polynomials. The inverse matrices $M(r)^{-1}$ are the coefficient arrays of these polynomials, with $P_n(x;r)=\sum_{k=0}^n (M(r)^{-1})_{n,k}x^k$.
\end{proposition}
\begin{proof}
We have that $M(r)$ is given by
\begin{align*}
M(r)&=\left((1-rx), x(1-rx)\right)^{-1}\cdot \left(\frac{1}{1-rx}, \frac{x}{1-rx}\right)\\
&=(c(rx), x c(rx))\cdot \left(\frac{1}{1-rx}, \frac{x}{1-rx}\right)\\
&=\left(\frac{c(rx)}{1-rxc(rx)}, \frac{x c(rx)}{1-r x c(rx)}\right)\\
&=\left(c(rx)^2, x c(rx)^2\right).\end{align*}
Thus we have 
$$M(r)_{n,k}=\frac{2(k+1)}{n+k+2} \binom{2n+1}{n-k}r^{n-k}.$$ 
The production matrix of $M(r)$ begins 
$$\left(
\begin{array}{cccccc}
 2 r & 1 & 0 & 0 & 0 & 0 \\
 r^2 & 2 r & 1 & 0 & 0 & 0 \\
 0 & r^2 & 2 r & 1 & 0 & 0 \\
 0 & 0 & r^2 & 2 r & 1 & 0 \\
 0 & 0 & 0 & r^2 & 2 r & 1 \\
 0 & 0 & 0 & 0 & r^2 & 2 r \\
\end{array}
\right).$$
Then the inverse matrix 
$$M(r)^{-1}=\left(\frac{1}{(1+rx)^2}, \frac{x}{(1+rx)^2}\right)$$ is the coefficient array of the family of orthogonal polynomials $P_n(x;r)$ that satisfy the three-term recurrence 
$$P_n(x;r)=(x-2r) P_{n-1}(x;r)-r^2 P_{n-2}(x;r),$$ with 
$P_0(x;r)=1$ and $P_1(x;r)=x-2r$.
\end{proof}
We finally note that the matrices 
$$((1-x)^k, x(1-x)^k)^{-1}\cdot \left(\frac{1}{1-x}, \frac{x}{1-x}\right)$$ produced by the $k$-th production matrices of $B$ are the matrices $A^{k,0}$ of \cite{Drube}

\section{Conclusions} We have shown that by passing to the higher production matrices (as described in this note) we are led to further Riordan arrays of interest. In this way, for each Riordan array, we have produced a canonical sequence of Riordan arrays associated to it by this production matrix process.  We note that in the special case of the identity matrix, the process once again produces the identity matrix at each stage. More generally, we see that the Appell subgroup (that is, Riordan arrays of the form $(g(x),x)$) is invariant under this process.

An alternative process that associates a sequence of Riordan arrays to each Riordan array is to take the matrix produced by the second production matrix, and then to repeat this operation on that matrix. Thus we always take the second production matrix of the current matrix. For instance, starting with the binomial matrix 
$$\left(\frac{1}{1-x}, \frac{x}{1-x}\right)=\left(\frac{1}{1+x}, \frac{x}{1+x}\right)^{-1},$$ we get the following sequence of matrices:
\begin{scriptsize}
$$\left(\frac{1}{1+x}, \frac{x}{1+x}\right)^{-1} \to \left(\frac{1}{(1+x)^2}, \frac{x}{(1+x)^2}\right)^{-1} \to \left(\frac{1}{(1+x)^4}, \frac{x}{(1+x)^4}\right)^{-1} \to \left(\frac{1}{(1+x)^8}, \frac{x}{(1+x)^8}\right)^{-1} \to \cdots.$$
\end{scriptsize}

\bigskip
\hrule
\bigskip
\noindent 2020 {\it Mathematics Subject Classification}: Primary
11C20; Secondary 11B83, 15A21, 15B36
\noindent \emph{Keywords:}  Riordan array, Riordan group, production matrix.

\bigskip
\hrule
\bigskip
\noindent (Concerned with sequences
\seqnum{A000108},\seqnum{A007318}, \seqnum{A033184}, \seqnum{A085478}, and \seqnum{A092276}.)

\end{document}